\documentclass[12pt,oneside,reqno]{amsart}

\usepackage{amssymb}
\usepackage[active]{srcltx}
\usepackage{a4wide}
\usepackage{amsmath,amsthm}
\usepackage{amssymb}
\usepackage{amstext}
\everymath{\displaystyle}
\usepackage{cite}
\usepackage{epsfig}
\usepackage{enumerate}
\usepackage{color}
 \usepackage{setspace} 
\usepackage{amssymb}
\newtheorem{theorem}{Theorem}[section]
\newtheorem{lemma}[theorem]{Lemma}
\newtheorem{proposition}[theorem]{Proposition}
\theoremstyle{definition}
\newtheorem{definition}[theorem]{Definition}
\newtheorem{example}[theorem]{Example}

\theoremstyle{remark}
\newtheorem{remark}[theorem]{Remark}
\newtheorem{note}[theorem]{Note}
\newtheorem{result}{Result}
\usepackage[linkcolor=blue, urlcolor=blue, citecolor=blue,
colorlinks, bookmarks]{hyperref}

\newcommand{\be}{\begin{equation}}
\newcommand{\ee}{\end{equation}}

\begin{document}

\title[Set-valued fractal approximation ]{Set-valued fractal approximation for countable data sets}



\author{Parneet Kaur}
\address{Department of Mathematics, 
Punjab Engineering college
(Deemed to be University), 
Sector 12 Chandigarh 160012, India}
\email{parneetkaur.phd24maths@pec.edu.in}

\author{Rattan Lal}
\address{Department of Mathematics, 
Punjab Engineering college
(Deemed to be University), 
Sector 12 Chandigarh 160012, India}
\email{rattanlal@pec.edu.in}

\author{Ankit Kumar}
\address{Department of Mathematics, 
Punjab Engineering college
(Deemed to be University), 
Sector 12 Chandigarh 160012, India}
\email{ankitkumar@pec.edu.in}
\author{Saurabh Verma}
\address{Department of Applied Sciences, IIIT Allahabad, Prayagraj, India 211015}
\email{saurabhverma@iiita.ac.in}






 

\begin{abstract} 
Fractal geometry deals mainly with irregularity and captures the complexity of a structure or phenomenon. In this article, we focus on the approximation of set-valued functions using modern machinery on the subject of fractal geometry. We first provide a construction of fractal functions for countable data sets and use these functions in the approximation and study of set-valued mappings. We also show the existence and uniqueness of an invariant Borel measure supported on the graph of a set-valued fractal function. In addition, we obtain some effective bounds on the dimensions of the constructed set-valued fractal functions.

\end{abstract}

\maketitle



\section{Introduction}\label{section 1}
Approximation theory in fractals is one of the most celebrated topics in the literature. Firstly, Barnsley \cite{MF1} introduced the notion of fractal interpolation functions (FIFs) for a finite data set based on the theory of iterated function systems and the Banach contraction theorem. In continuation of this foundational result of Barnsley, Navascu\'es \cite{Navascues2010,Navascues2005} explored a parameterized class of FIFs, related with the continuous function defined on the compact interval of $\mathbb{R}$, known as $\alpha$-fractal function. The $\alpha$- fractal function is then explored more in \cite{JV1,sverma4, GCV1}, whereas the generalized $C^r$ fractal function is investigated in \cite{B,Massopust2016fractal,Navascues2004}.

The authors in \cite{Ver21} have generalized the concept of a $\alpha$-fractal function for the set-valued maps (SVMs). The approximation of the set-valued surfaces is studied in \cite{AV3}. SVMs are widely applied in various fields including control theory, operational research, mathematical modeling, game theory, etc. For a better insight into the fundamental aspect of SVMs in more detail, refer to \cite{levin}. The algebra on sets is not the same as the algebra on numbers. The binary metric linear combination of sets is used in \cite{Artstein}, is further extended in \cite{Berdysheva} called the metric sum of the linear combination of sets, and the Minkowski sum of two sets is used in \cite{Ndyn} and the sum defined in \cite{Artstein} has been extended in \cite{Berdysheva} known as a metric sum of the linear combination of sets. Throughout the paper, the sum operation of sets is taken as the Minkowski sum of sets.

In 1986, Barnsley \cite{MF1} introduced FIFs on the finite data set. 
Firstly, the concept of FIFs on finite data sets was generalized to countable FIFs \cite{Secelean} on countable data sets by Secelean \cite{Secelean2}. In \cite{CAV1}, Chandra et al. extended the result of Secelean \cite{Secelean2} to different types of FIFs on countable data sets and Miculescu et al. \cite{Miculescu} to a general data set using the Banach contraction theory. Recently, Verma and  Priyadarshi,\cite{Mverma1 , Mverma3} have shown the existence of FIFs on the general data set and bivariate FIFs for countable data set using the Rakotch contraction theory. In this paper, we have studied the construction of FIFs on countable data set. 

The countable iterative function system (CIFS) is more general than the finite iterative function system (IFS). Firstly, CIFS was studied by Mauldin and Urbanski \cite{Mauldin96} for conformal contractions. Hausdorff dimension of the attractor of CIFS that contains similarity and conformal contractions is studied in \cite{Mauldin96, Hille, Jiang}. One can refer, \cite{Mauldin96, Secelean } to see the example in which fractals constructed by the CIFS cannot be constructed by IFS.
In this paper, the CIFS we have taken consists of bi-Lipschitz contraction maps, which are more generalized than the self-similar and conformal maps.  

Several theories for the approximation of SVMs are given in the literature. For example, in \cite{levin}, where the concept of univariate data interpolation functions in general metric space is given and in \cite{Vitale}, convex SVMs are approximated using the set-valued Bernstein polynomial. For an in-depth exploration of the approximation of convex SVMs, refer to \cite{Rbaier, Mcampiti, Ndyn1} and for the approximation of compact SVMs, refer to \cite{Berdysheva, Ndyn}. In this paper, we study the approximation of fractals on SVMs.

The concept of fractal dimension always captivates the attention of scientists. The dimension of the fractals can be studied by various methods such as Hausdorff dimension, box counting dimension, packing dimension, entropy dimension, quantization dimension, and many more. For a more detailed study of the fractal dimension, see \cite{AchourBilel, Mauldin88, Chandra22, Liang10, PostOper12, AV2, DV1, GCV2, BilelLiang, Yu}. In this paper, we obtain the bounds on the Hausdorff dimension of the graph of the $\alpha$-fractal function using the covering method.

This paper presents a new idea for constructing the fractal function of SVMs on countable data sets. We further showed the existence of an invariant Borel probability measure supported by the graph of an $\alpha$-fractal function of SVMs. We also obtained the bounds of the Hausdorff dimension on the fractal function of the constructed set-valued fractal functions.


\section{Preliminaries}
\begin{definition} \cite{Fal}
    Let $(X,d)$ be a metric space and $E$ be a subset of $X$. The Hausdorff dimensions of $E$ is defined as 
    \[  \text{dim}_H E= \text{inf} \big\{s>0: H^s(E)=0\big\} = \text{sup}\big\{ s >0: H^s(E) =\infty\big\}\]
    and the $H^s(E)$ is s-dimensional Hausdorff measure of $E,$ is defined as:
    \[ H^s(E)=\lim_{\delta \rightarrow 0}H_\delta^s(E) \text{ and }  H_\delta^s(E) = \text{inf} \bigg\{ \sum_{i=1}^\infty\lvert U_i \rvert^s : 0< \lvert U_i\rvert \le \delta \text{ and } E\subset \bigcup_{i=1}^{\infty} U_i \bigg\}  \]
    where $\lvert U_i \rvert$ represents the diameter of $U_i$. The diameter of $U$ is defined as: \[\lvert U\rvert =\sup\left\{d(t,w): t,w \in U\right\}.\]
\end{definition}

 \begin{definition}\cite[Definition 1.3.1]{Aubin}
 Let $F:X\rightrightarrows Y$ be an SVM from the metric space $X$ to the metric space $Y$. Consider
 \begin{equation}\label{Gf1}
   G_F=\{(t,w) \in X \times Y: w \in F(t)\}.
 \end{equation}
 We call $G_F$  the graph of SVM $F$. The value or image of $F$ at $t$ is represented by $F(t)$. The map $F$ is considered nontrivial if there exists at least one $t \in X$ for which $F(t)$ contains at least one element. Furthermore, $F$ is called strict if $F(t)$ contains at least one element for each $t \in X.$ The domain of $F$ is defined as 
  \[\text{Dom}(F):=\{t \in X: F(t) \ne \emptyset\}\] 
 and the range of $F$ is defined as \[\text{Im}(F):= \underset{t \in X}{\bigcup}F(t).\]
 \end{definition}
 
Throughout the paper, we denote $I$ as a closed and bounded interval of $\mathbb{R}.$
\begin{remark}\cite{Aubin}
 If the graph of a SVM is closed (or convex), then $F$ is defined to be closed (or convex).
 \end{remark}

 \begin{definition}\label{order}
 Consider $F_1,F_2$ be SVMs such that $F_1, F_2: I\rightrightarrows \mathbb{R}$. Then, $F_1\leq F_2$ if and only if $F_1(t)\subseteq F_2(t)$ for all $t\in I$.
 \end{definition}
 \begin{lemma}\emph{\cite{Aubin}}
Assume $F$ to be an SVM from the normed space $X$ to the normed space $Y$, which is said to be convex if and only if $\text{for all } t, w \in \text{Dom}(F) \text{and}~\theta \in [0,1]$, we have
 \[\theta F(t)+(1-\theta)F(w) \subset F\left(\theta t+(1-\theta)w\right).\]
\end{lemma}

 \begin{definition}\cite{Aubin}
 Assume $F $ be a SVM from $X$ to $Y$. Let $t \in \text{Dom}(F)$ provided that for each neighborhood $N$ of $F(t),$ the following holds
 \begin{equation}\label{uppersemi}
 \text{there exists } \delta > 0  \text{ such that } F(t') \subset N \text{ for all } t' \in B_X(t, \delta).
 \end{equation}
 Then, $F$ is said to be upper semicontinuous at $t$. If this condition holds for every $t\in \text{Dom}(F)$, in that case $F$ is called as an upper semicontinuous function.
 
 For the lower semicontinuity of $F$ at $t$. Let N be an open subset of Y such that $N\subset Y$ with $N\cap F(t) \neq \emptyset $ and the following holds 
 \[ \text{ there exist } \delta >0 \text{ such that } F(t') \cap N \neq \emptyset \text{ for all } t' \in B_X(t,\delta). \]
 Then, $F$ is said to be lower semicontinuous at $t$. If this condition holds for every $t\in \text{Dom}(F),$ in that case $F$ is called as a lower semicontinuous function. If $F$ is both upper semicontinuous and lower semicontinuous then $F$ is continuous.

 \end{definition}
 Let $(X,d)$ be a complete metric space, and the notation $\mathrm{K}(X)$ is used to denote the collection of all non-empty compact subsets of $X$. Consider $E$ and $E'$ belongs to $\mathrm{K}(X)$ and $H_d$ be the Hausdroff metric on $\mathrm{K}(X)$. Defined as:
 \[H_d(E,E')=\text{inf}\{\delta>0,E'\subset E_{\delta} ~\text{and}~ E \subset E'_{\delta} \} \]
 where $E_{\delta}$ and $E'_{\delta}$ represents the $\delta$-neigbourhoods of sets $E$ and $E'$, respectively. It is well established result that $(\mathrm{K}(X),H_d)$ forms a complete metric space.
 \begin{remark}
\cite[Proposition $1.17$]{Hu1997handbook}\label{new71}
   Let us recall some properties of the Hausdorff metric.
\begin{itemize}
  \item For any $k \in \mathbb{R}$ and $S_1,S_2 \in  \mathrm{K} (\mathbb{R})$, we have $H_d (k S_1,k S_2)=  \lvert k\rvert   H_d (S_1,S_2).$
  \item Assume that $X$ is a normed space. For any $S_1, S_2, S_3, S_4 \in  \mathrm{K} (\mathbb{R}),$ we have
   \[H_d (S_1 + S_2, S_3 + S_4) \le  H_d (S_1,S_3) + H_d (S_2, S_4),\]
   where sum of sets is taken as Minkowski sum. For sets $C,D \in \mathrm{K} (\mathbb{R}),$ the Minkowski sum  is defined as  \[ C+D := \{c + d : c \in  C  , d \in D\}. \]
\end{itemize}
\end{remark}
 
\begin{definition}\cite{Hut}\label{Iterated}
Let $(X,d)$ be a compact metric space. Let $w_i:(X,d)\to (X,d)$ be contraction mappings for each $i\in \mathbb{N},$ i.e., there exist $c_i <1$ such that \[ d(w_i(t),w_i(w))\le c_id(t,w) ~~\forall~~ t,w \in X  \text{ and } i \in \mathbb{N},\] Then the system $\{(X,d);~w_i,~i\in \mathbb{N}\} $ is known as countable iterated function system (CIFS).
\end{definition}
\begin{definition}
    Consider $\{(X,d);~w_i,~i\in \mathbb{N}\} $ to be a CIFS. If $K \subseteq X$ and $K\neq \emptyset$, also if $K$ satisfies the \[ K= \overline{\underset{i\in\mathbb{N}}{\bigcup}w_i(K)},\] then the set $K$ is known as the attractor of the CIFS.
\end{definition}

\begin{definition}\cite{Hut}\label{OSC}
A CIFS, $\big\{(X,d);~w_i,~i\in\mathbb{N}\big\}$ is said to satisfy the Open Set Condition (OSC), if there exists a non-empty open set $U\subset \mathbb{R}$ such that 
\[w_i(U)\subset U ~\forall ~ i \in \mathbb{N},~~ w_i(U)\cap w_j(U)=\emptyset ~\text{ for }~ i\neq j.\]
In addition, if $K$ is an attractor of CIFS and satisfies the $U\cap K\neq \emptyset$, then the CIFS is said to satisfy the Strong Open Set Condition (SOSC).
 \end{definition}
 \begin{result} \label{res 4.5}
If $A$ is dense subset of $\mathbb{R}$ and $f: \mathbb{R}\to \mathbb{R}$ is a continuous function, then
\begin{enumerate}
    \item if $f(t)\leq 0$ for each $t\in A$, then $f(t)\leq0$ for each $t\in 
    \mathbb{R}.$
    \item if $f(t)\geq0$ for each $t\in A$, then $f(t)\geq0$ for each $t\in 
    \mathbb{R}.$
\end{enumerate}
\end{result}
Throughout the paper, we have used the following notation : $\mathrm{K}(\mathbb{R})$ for the collection of all non-empty compact subsets of $\mathbb{R}$, $\mathrm{K}_c(\mathbb{R})$ for the collection of all non-empty convex subsets of $\mathbb{R},$ and $\mathcal{C}(I, \mathrm{K}(\mathbb{R}))$ for the collection of all continuous mappings from a closed and bounded interval $I$ to $\mathrm{K}(\mathbb{R}).$

\section{Construction of Fractal Functions in $C( I, \mathrm{K}( \mathbb{R}) )$}\label{sec3}
Let $\mathcal{C}( I, \mathrm{K}( \mathbb{R}) )$ denote the space endowed with the metric $\mathfrak{d}_{\mathcal{C}}$, where 
\begin{equation*}
    \mathfrak{d}_{\mathcal{C}}(S_1,S_2)=\lVert S_1-S_2 \rVert _{\infty}=\sup_{\substack{t\in I}}H_d( S_1(t),S_2(t)).
          \end{equation*}
With this metric $\mathfrak{d}_{\mathcal{C}}$, the space $  \mathcal{C}( I, \mathrm{K}( \mathbb{R}) ) $ is  complete.

\begin{theorem}\label{Thm3.2}
Let   $\vartheta:=\{(t_1,t_2,t_3,\ldots,t_\infty=\lim t_n): t_1 < t_2 <t_3< \cdots <t_\infty=\lim t_n \},$ which  defines a partition of $I,$ and let $I_n=[t_n,t_{n+1}]$ and assume the following contraction homeomorphism $\zeta_n:I \rightarrow I_n $ such that $\zeta_{n}(t_1)=t_{n} \text{ and } \zeta_{n}(t_\infty)=t_{n+1}$ or $\zeta_{n}(t_1)=t_{n+1} \text{ and } \zeta_{n}(t_\infty)=t_{n}$. Let $\Phi\in \mathcal{C}(I,\mathrm{K}(\mathbb{R}))$  and the base function $B \in \mathcal{C}( I, \mathrm{K}( \mathbb{R}) )$ satisfies the following condition
   \[B(t_1)-\Phi(t_1)=B(t_\infty)-\Phi(t_\infty),\]
  where $\mathcal{A}_1-\mathcal{A}_2=\left\{a_1-a_2: a_1\in \mathcal{A}_1, a_2 \in \mathcal{A}_2 \text{ and } \mathcal{A}_1,\mathcal{A}_2\in \mathrm{K}(\mathbb{R}) \right\}$. 
  If $\alpha \in \mathbb{R}$ is a scaling factor, satisfying the $\lvert \alpha \rvert < 1,$ then  a unique function $\Phi^{\alpha}_{\vartheta,B} \in \mathcal{C}( I, \mathrm{K}( \mathbb{R}) )$ exists,  satisfying the following self-referential equation 
   \begin{equation}\label{self12}
      \Phi^{\alpha}_{\vartheta,B}(t)= \Phi(t) + \alpha [\Phi^{\alpha}_{\vartheta,B}(\zeta_n^{-1}(t)) - B(\zeta_n^{-1}(t))] ~\text{ for every }~ t\in I_n,
   \end{equation}
    where $n \in J=\{1,2,3\ldots\}. $ and for $t_\infty$ satisfies the
    \begin{equation}\label{eqn4}
        \Phi^{\alpha}_{\vartheta,B}(t_\infty)=\underset{n \rightarrow \infty}{\lim}\Big(\Phi(t_n) + \alpha [\Phi^{\alpha}_{\vartheta,B}(\zeta_n^{-1}(t_n)) - B(\zeta_n^{-1}(t_n))]\Big).
    \end{equation}
    \
   \end{theorem}
   \begin{proof}
    Define the set $ \mathcal{C}_\Phi(I,\mathrm{K}( \mathbb{R}))= \{ U \in \mathcal{C}(I,\mathrm{K}( \mathbb{R})): U(t_1)-B(t_1)~= U(t_\infty)-B(t_\infty)\}.$ It is straightforward that $\mathcal{C}_\Phi(I,\mathrm{K}( \mathbb{R}))$ forms a closed subset within the space $\mathcal{C}(I,\mathrm{K}( \mathbb{R})).$  Thus, $ \mathcal{C}_\Phi(I,\mathrm{K}( \mathbb{R}))$ is  complete with the  metric $\mathfrak{d}_{\mathcal{C}}.$
   Now, we define the  \textit{Read-Bajraktarevi\'c }(RB) operator $\phi$ on $ \mathcal{C}_\Phi( I, \mathrm{K}( \mathbb{R}) )  ,$ given by 
   \[(\phi U)(t)= \Phi(t) + \alpha  [U(\zeta_n^{-1}(t)) - B(\zeta_n^{-1}(t))]\]
   for each $t \in I_n $ and where $n \in J $. For the $t_\infty$ is given by \[ (\phi U)(t_\infty)=\underset{n \rightarrow \infty}{\lim} \Big(\Phi(t_n) + \alpha  [U(\zeta_n^{-1}(t_n)) - B(\zeta_n^{-1}(t_n))\Big) .\]
   The operator $\phi$ is well defined, which follows from the assumptions on  $\Phi,B$, and $\alpha$. Let $U,V \in \mathcal{C}_\Phi( I, \mathrm{K}( \mathbb{R}) ) ,$  now using the properties of the Hausdorff metric given in the  Remark \ref{new71} for any $t\in I_n$, we get 
   \begin{align*}
     H_d((\phi U)(t)&,(\phi V)(t))\\
&= H_d\Big(\Phi(t) + \alpha  [U(\zeta_n^{-1}(t)) - B(\zeta_n^{-1}(t))],\Phi(t) + \alpha  [V(\zeta_n^{-1}(t))-B(\zeta_n^{-1}(t))]\Big)\\
     &\leq H_d\Big( \alpha U(\zeta_n^{-1}(t)),\alpha V(\zeta_n^{-1}(t))\Big)\\
     &=\lvert \alpha \rvert H_d\Big(  U(\zeta_n^{-1}(t)), V(\zeta_n^{-1}(t))\Big).\\
   \end{align*}
   For $t=t_\infty,$ we have
   \begin{align*}
        &H_d((\phi U)(t_\infty),(\phi V)(t_\infty))\\ &=\underset{n \rightarrow\infty}{\lim}H_d\Big(\Phi(t_n) + \alpha  [U(\zeta_n^{-1}(t_n)) - B(\zeta_n^{-1}(t_n))],\Phi(t_n) + \alpha  [V(\zeta_n^{-1}(t_n))-B(\zeta_n^{-1}(t_n))]\Big)\\ &\le \underset{n \rightarrow\infty}{\lim}H_d\Big( \alpha U(\zeta_n^{-1}(t_n)),\alpha V(\zeta_n^{-1}(t_n))\Big)\\  & = \underset{n \rightarrow\infty}{\lim}  \lvert \alpha \rvert H_d\Big(  U(\zeta_n^{-1}(t_n)), V(\zeta_n^{-1}(t_n))\Big).
   \end{align*}
   From the above two inequalities we can conclude that,
   \begin{align*}
       H_d((\phi U)(t),(\phi V)(t)) &\le \lvert \alpha \rvert\sup_{t \in I} H_d( U(t),V(t))\\ 
     &=\lvert \alpha \rvert\lVert U-V\rVert _{\infty}.
   \end{align*}
Because, $ \lvert \alpha \rvert\lVert U-V\rVert _{\infty} $ is independent of $t,$ we get
   \begin{equation*}
    \lVert \phi U-\phi V\rVert _{\infty}
     \le   \lvert \alpha \rvert\lVert U-V\rVert _{\infty}.  
   \end{equation*}
  Since $\lvert \alpha \rvert < 1,$ the mapping $\phi$ acts as a contraction on $\mathcal{C}_\Phi( I, \mathrm{K}( \mathbb{R}) )$. Thus, according to the Banach contraction principle, we conclude that $\phi$ possesses a unique fixed point, say $\Phi^{\alpha}_{\vartheta,B}$ in $\mathcal{C}_\Phi(I,\mathrm{K}(\mathbb{R}))$. Then $\Phi^{\alpha}_{\vartheta,B}$  satisfies the following self-referential equation,
  \[\Phi^{\alpha}_{\vartheta,B}(t)= \Phi(t) + \alpha  [\Phi^{\alpha}_{\vartheta,B}(\zeta_n^{-1}(t)) - B(\zeta_n^{-1}(t))]\]
   for every  $n \in J=\{1,2,3\ldots\}$ and $t \in I_n$ and for the  $t_\infty$ is given by
   \[  \Phi^{\alpha}_{\vartheta,B}(t_\infty)=\underset{n \rightarrow \infty}{\lim}\Big( \Phi(t_n) + \alpha [\Phi^{\alpha}_{\vartheta,B}(\zeta_n^{-1}(t_n)) - B(\zeta_n^{-1}(t_n))]\Big).\]
   
   \end{proof}
  
\begin{note}\label{unchange}
We will use the following notations as:
\begin{itemize}
\item Set difference as following:
    \[\mathcal{A}_1-\mathcal{A }_2=\left\{a_1-a_2: a_1\in \mathcal{A}_1, a_2 \in \mathcal{A}_2 \text{ and } \mathcal{A}_1,\mathcal{A}_2\in \mathrm{K}(\mathbb{R})\right\}.\]
    \item Unless otherwise stated, we assume that $J$, $B$, and $\vartheta$ are the same as each of them stated in  Theorem \ref{Thm3.2}.
    \item In the absence of ambiguity, we take $\Phi^{\alpha} $ instead of  $\Phi^{\alpha}_{\vartheta,B}$.
\end{itemize}

\end{note}
  
\begin{remark}\label{rmrk3.3}
With respect of \eqref{self12}, we obtain
\[\Phi^{\alpha}(t_i)= \Phi(t_i) + \alpha  \Phi^{\alpha}(t_1) - \alpha B(t_1)=\Phi(t_i) + \alpha  \Phi^{\alpha}(t_\infty) - \alpha B(t_\infty) \text{ for every } t_i\in \vartheta, \] 
where $ i = 1,2,3,\dots~ $. Further, if $\Phi^{\alpha}$ and base function $B$ takes the single-valued at the boundary points in a way that $\Phi^{\alpha}(t_1) -  B(t_1)= \Phi^{\alpha}(t_\infty) -  B(t_\infty)=\{0\}$ then
\begin{equation*}
    \begin{aligned}
           &\Phi^{\alpha}(t_i)=\Phi(t_i)\text{ for each } i=1,2,3,\ldots~.
 \end{aligned}
 \end{equation*}
 This establishes that $\Phi^{\alpha}$ is a set-valued countable fractal interpolation function (CFIFs).
\end{remark}

\begin{note}\label{Nt3.4}
The preceding remark suggests the next: When  \[\Phi(t_1)-B(t_1)=\Phi(t_\infty)-B(t_\infty)=\{0\},\] that is, at the end points $\Phi$ and $B$ are single valued. Consider
\[\mathcal{C}_\Phi(I,\mathrm{K}( \mathbb{R}))= \Big\{ U \in \mathcal{C}(I,\mathrm{K}( \mathbb{R})): U(t_1)-B(t_1)= U(t_\infty)-B(t_\infty)=\{0\}\Big\}.\] The set defined above is a complete metric space, and as stated in Theorem \ref{Thm3.2}, the RB operator $\phi:\mathcal{C}_\Phi(I,\mathrm{K}( \mathbb{R})) \to \mathcal{C}_\Phi(I,\mathrm{K}( \mathbb{R}))$ is evidently well defined and satisfies the properties of a contraction mapping. Consequently, we obtain a unique fixed point of $\phi$, take it as $\Phi^\alpha$ satisfying the \[ \Phi^{\alpha}(t_i)=\Phi(t_i) \text{ for all } i=1,2,3,\ldots ~.\] Thus, $\Phi^{\alpha}$ is a set-valued countable fractal interpolation function (CFIFs).
\end{note}
In the next example, we present a  base function $B \in \mathcal{C}( I, \mathrm{K}( \mathbb{R}) )$ that satisfies the end point condition $B(t_1)-\Phi(t_1)=B(t_\infty)-\Phi(t_\infty).$  

\begin{example}
   
    Let $h:I \to \mathbb{R}$ be a continuous function such that $h(t_1)=1,~h(t_\infty)=1.$ Base function is
        \[ B(t)=h(t)\Phi(t)+(t-t_1)(\Phi(t_\infty)-\Phi(t_1))+(t_\infty -t)(\Phi(t_1)-\Phi(t)). \]
    
   \end{example}

Observe that the function $\Phi^{\alpha}$ is a parametric function determined by the following parameters: the base function $B$, the function $\Phi$ itself, the partition $\vartheta$, and the scaling function $\alpha$. To study the collective behavior of function $\Phi^{\alpha}$ with respect to varying parameters, we introduce a set-valued mapping on a countable dataset, denoted as $\mathfrak{F}^{\alpha}_{B}: \mathcal{C}(I, \mathrm{K}(\mathbb{R})) \rightarrow \mathcal{C}(I, \mathrm{K}(\mathbb{R}))$ which is defined by 
\begin{equation}\label{fractope}
    \mathfrak{F}^{\alpha}_{B}(\Phi)=\Phi^{\alpha},~ \text{where}~ \alpha \in (-1,1).
\end{equation}
We call this map $\mathfrak{F}^{\alpha}_{B} $ as a $\alpha$-countable  fractal operator.

\begin{remark}
The $\alpha$-fractal operator on the SVM is defined in \cite{Ver21}. The fractal operator is widely studied for the single-valued maps and bivariate single-valued map. For better understanding, see, \cite{Navascues2005,Navascues2010,sverma3}. Here, we have introduced the concept of the set-valued $\alpha$-fractal operator corresponding to the countable data set and we call it the $\alpha$-countable fractal function. 
   
\end{remark}
In the next theorem, we prove the continuity of $\mathfrak{F}^{\alpha}_B$ ($\alpha$-countable fractal operator) as defined in \eqref{fractope}.

\begin{theorem}\label{fractoperator}
The $\alpha$-countable fractal operator $\mathfrak{F}^{\alpha}_B$ is a continuous map.
\end{theorem}
\begin{proof}
Consider a convergence sequence $\{\Phi_k\}$ in $\mathcal{C}(I,\mathrm{K}(\mathbb{R}))$, such that, $\{\Phi_k\}$ converges to $\Phi\in \mathcal{C}(I,\mathrm{K}(\mathbb{R})).$  \\To show, $\mathfrak{F}^{\alpha}_B $ is a continuous function, we will show that $\Phi_k^{\alpha}\to \Phi^{\alpha} \text{ as }k\rightarrow\infty.$\\ Since, sequence $\{\Phi_k\}$ converges to $\Phi$, for $\epsilon>0,$ there exists a $N\in \mathbb{N}$ such that
\begin{align*}
    \lVert \Phi_k-\Phi\rVert_{\infty}< \epsilon(1-\lvert\alpha \rvert ) \text{  for all } k \geq N
\end{align*}
Therefore,
\begin{align*}
    \underset{t\in I}{\sup}H_d(\Phi_k(t), \Phi(t)) < \epsilon(1-\lvert\alpha\rvert ) \text{ for all } k \geq N.
    \end{align*}
  For the $t\in I_j$, the above give us that
    \begin{align*}
        H_d(&\Phi_k^{\alpha}(t), \Phi^{\alpha}(t))\\ =&H_{d}\bigg(\Phi_{k}(t)+\alpha\left[\Phi^{\alpha}_{k}(\zeta^{-1}_{j}(t))-B(\zeta^{-1}_{j}(t))\right], \Phi(t)+\alpha\left[\Phi^{\alpha}(\zeta^{-1}_{j}(t))-B(\zeta^{-1}_{j}(t))\right]\bigg)\\
\leq & H_{d}(\Phi_{k}(t),\Phi(t))+\lvert \alpha \rvert H_d(\Phi^{\alpha}_{k}(\zeta^{-1}_{j}(t)),\Phi^{\alpha}(\zeta^{-1}_{j}(t))).
   \end{align*}
   If $t=t_\infty $, we have the following
    \begin{align*}
         H_d(\Phi_k^{\alpha}(t_\infty), \Phi^{\alpha}(t_\infty)) =& \underset{j\rightarrow \infty}{\lim}H_{d}\bigg(\Phi_{k}(t_j)+\alpha\left[\Phi^{\alpha}_{k}(\zeta^{-1}_{j}(t_j))-B(\zeta^{-1}_{j}(t_j))\right], \\&\hspace{3cm}\Phi(t_j)+\alpha\left[\Phi^{\alpha}(\zeta^{-1}_{j}(t_j))-B(\zeta^{-1}_{j}(t_j))\right]\bigg)\\
\leq & \underset{j\rightarrow \infty}{\lim}\Big( H_{d}(\Phi_{k}(t_j),\Phi(t_j))+\lvert \alpha \rvert H_d(\Phi^{\alpha}_{k}(\zeta^{-1}_{j}(t_j)),\Phi^{\alpha}(\zeta^{-1}_{j}(t_j)))\Big)\\ =&  H_{d}(\Phi_{k}(t_\infty),\Phi(t_\infty))+\underset{j\rightarrow \infty}{\lim}\lvert \alpha \rvert H_d(\Phi^{\alpha}_{k}(\zeta^{-1}_{j}(t_j)),\Phi^{\alpha}(\zeta^{-1}_{j}(t_j))).
    \end{align*}
    From above two inequality, we can conclude
    \begin{align*}
         \underset{t\in I}{\sup}~H_d(\Phi^{\alpha}_{k}(t), \Phi^{\alpha}(t)) &\leq \frac{1}{1-\lvert \alpha \rvert}~\underset{t\in I}{\sup}~H_d(\Phi_{k}(t),\Phi(t)).
        \end{align*}
        This implies,
        \begin{align*}
            \lVert \Phi_k^{\alpha} (t)- \Phi^\alpha(t)\rVert _\infty < \epsilon \text{ for all } k\geq N \text{ and } t \in I.
        \end{align*}
This establishes the proof.

\end{proof}
\section{Some Approximation Aspect}\label{sec4}

In previous section, we have seen that $\Phi^{\alpha}$ is defined by the following self-referential equation:
 \[\Phi^{\alpha}(t)= \Phi(t) + \alpha  \left[\Phi^{\alpha}\left(\zeta_n^{-1}(t)\right) - B\left(\zeta_n^{-1}(t)\right)\right] \text{ for every }  t \in I_n,\]
where $n \in J =\{1,2,3,\ldots\}$ and for the $t_\infty$ is given by
 \[  \Phi^{\alpha}(t_\infty)=\underset{n \rightarrow \infty}{\lim}\Big( \Phi(t_n) + \alpha [\Phi^{\alpha}(\zeta_n^{-1}(t_n)) - B(\zeta_n^{-1}(t_n))]\Big).\]

In this section we will study the approximation of the set-valued fractal function taken over the countable data set.
 
For the following proposition, we took the motivation from the \cite[Proposition $1$]{Ver21} in which the error bound between a SVM and its associated $\alpha$-fractal function is calculated. In the literature \cite{Navascues04,sverma2020}, the error bound between a univarite and bivarite single valued map and their associated $\alpha$-fractal function is calculated. In the following proposition, we have found the error bound between a SVM and $\alpha$-countable fractal function.

 \begin{proposition}\label{new79}
 For any $\Phi\in \mathcal{C}(I,\mathrm{K}(\mathbb{R}))$ and its corresponding $\Phi^{\alpha}.$ The error bound is given by:
   \[\lVert \Phi^{\alpha} -\Phi\rVert _{\infty} \leq \frac{\lvert \alpha \rvert}{1- \lvert \alpha \rvert} \lVert \Phi-B \rVert _{\infty}+\frac{2 \lvert \alpha \rvert}{1- \lvert \alpha \rvert} \lVert \Phi \rVert_{\infty}.\]
 \end{proposition}
 \begin{proof}
Using the Hausdorff metric property given in Note \ref{new71} and self-referential equation, we get 
\begin{align*}
   & H_d(\Phi^{\alpha}(t), \Phi(t))\\
   =& H_d\bigg(\Phi(t) + \alpha  \left[\Phi^{\alpha}\left(\zeta_n^{-1}(t)\right) -                 B\left(\zeta_n^{-1}(t)\right)\right], ~\Phi(t)\bigg)\\ 
   \leq &  H_d\bigg( \alpha  \left[\Phi^{\alpha}\left(\zeta_n^{-1}(t)\right) -        B\left(\zeta_n^{-1}(t)\right)\right], ~\{0\}\bigg)=\lvert \alpha \rvert H_d\bigg(\Phi^{\alpha}\left(\zeta_n^{-1}(t)\right)- B\left(\zeta_n^{-1}(t)\right), ~\{0\}\bigg)\\
    \leq&  \lvert \alpha \rvert H_d\bigg(\Phi^{\alpha}\left(\zeta_n^{-1}(t)\right) -  B\left(\zeta_n^{-1}(t)\right),~ \Phi\left(\zeta_n^{-1}(t)\right)-\Phi\left(\zeta_n^{-1}(t)\right)\bigg)\\
    &\hspace{6cm}+\lvert \alpha \rvert H_d\bigg(\Phi\left(\zeta_n^{-1}(t)\right)-\Phi\left(\zeta_n^{-1}(t)\right), \{0\}\bigg)\\
    \leq &\lvert \alpha \rvert H_d\bigg(\Phi^{\alpha}\left(\zeta_n^{-1}(t)\right),  \Phi\left(\zeta_n^{-1}(t)\right)\bigg)+\lvert \alpha \rvert H_d\bigg(- B\left(\zeta_n^{-1}(t)\right),-\Phi\left(\zeta_n^{-1}(t)\right)\bigg)\\
    &\hspace{6cm}+2 \lvert \alpha \rvert H_d\bigg(\Phi\left(\zeta_n^{-1}(t)\right),\{0\}\bigg).
    \end{align*}
    Further taking the supremum over $t \in I_n$ where $n\in J=\{1,2,3\dots \}$. From this
     \begin{align*}
    H_d(\Phi^{\alpha}(t), \Phi(t)) &\leq \lvert \alpha \rvert \sup_{t\in I_n, n \in J}H_d\bigg(   \Phi^{\alpha}\left(\zeta_n^{-1}(t)\right),  \Phi\left(\zeta_j^{-1}(t)\right)\bigg)+\\ & \hspace{2cm}\lvert \alpha \rvert \sup_{t\in I_n, n \in J}
     H_d\bigg(B\left(\zeta_n^{-1}(t)\right),\Phi\left(\zeta_n^{-1}(t)\right)\bigg)\\
     & \hspace{2cm}+2 \lvert \alpha \rvert \sup_{t \in I_n, n \in J}H_d\bigg(  \Phi\left(\zeta_n^{-1}(t)\right),\{0\}\bigg)\\
     & \leq \lvert \alpha \rvert \lVert \Phi^{\alpha}-\Phi \rVert _{\infty}+ \lvert \alpha \rvert \lVert \Phi-B \rVert _{\infty}+2 \lvert \alpha \rvert  \lVert \Phi \rVert_{\infty}.
   \end{align*}
 Similar inequality we can show at the $t_\infty$ also. This implies,
 \begin{align*}
      \lVert \Phi^{\alpha}-\Phi \rVert _{\infty}\le  \lvert \alpha \rvert \lVert \Phi^{\alpha}-\Phi \rVert _{\infty}+ \lvert \alpha \rvert \lVert \Phi-B \rVert _{\infty}+2 \lvert \alpha \rvert  \lVert \Phi \rVert_{\infty}.
 \end{align*}
From this, we get the required error bound.
 \end{proof}
Before proceeding with the next result on constrained approximation in the context of set-valued fractal functions, let us first establish the following lemma.
 
\begin{lemma}\label{lem4.6}
 The set $S=\underset{k \in \mathbb{N}}{\bigcup}\left(\underset{1\leq i_1,\ldots,i_k< \infty}{\bigcup}\zeta_{i_{1}\ldots i_{k}}\bigg(\big\{t_1,t_2,t_3\ldots,t_\infty = \lim t_n \big\}\bigg)\right)$ is dense in interval $I=[0,1]$, where $\zeta_{i_{1}\ldots i_{k}}(t)=\zeta_{i_{1}}(\zeta_{i_{2}}(\ldots (\zeta_{i_{k}}(t))))$ and $k \in \mathbb{N}.$
\end{lemma}
\begin{proof}
For any point $t \in I$. We can obtain some $w_1\in \left\{t_1,t_2,t_3 \ldots,t_\infty\right\}$, such that $\lvert t-w_1 \rvert\leq \underset{i\in J}{\max} \left\{\frac{t_i-t_{i-1}}{2}\right\}$. Since $\zeta_i$ is a contraction mapping for each $i\in J$ with the contraction coefficient $a_i$. Choose $a=\underset{i\in J}{\max} \{a_i\}$, then for each $t \in I$ and for given $\epsilon>0$ we can choose $w_2\in \zeta_{i_{1}\ldots i_{k}}\big(\left\{t_1,t_2,t_3\ldots,t_\infty\right\}\big)$ for some $k\in \mathbb{N}$ such that,
\[\lvert t-w_2 \rvert\leq a^{k}\underset{i\in J}{\max} \left\{\frac{t_i-t_{i-1}}{2}\right\}<\epsilon.\]
This completes the proof.
\end{proof}

\begin{theorem}\label{Thm4.7}
Let $\Phi,\mathcal{U} \in \mathcal{C}(I , \mathrm{K}( \mathbb{R}))$ and $\vartheta$ as defined in Theorem \ref{Thm3.2}, and $\Phi(t_1),~\Phi(t_\infty),\\ ~\mathcal{U}(t_1),~\mathcal{U}(t_\infty)$ are single-valued. If $\Phi \le \mathcal{U}$, then $\Phi^{\alpha} \le \mathcal{U}^{\alpha}$ provided $B_\Phi, B_\mathcal{U} \in \mathcal{C}(I , \mathrm{K}( \mathbb{R}))$ satisfying $B_\Phi \le B_\mathcal{U}$ and $B_\Phi(t_1)=\Phi(t_1),B_\Phi(t_\infty)=\Phi(t_\infty),B_\mathcal{U}(t_1)=\mathcal{U}(t_1),B_\mathcal{U}(t_\infty)=\mathcal{U}(t_\infty).$
\end{theorem}
\begin{proof}
Let $B_\Phi, B_\mathcal{U} \in \mathcal{C}(I , \mathrm{K}( \mathbb{R}) )$ such that $B_\Phi \le B_\mathcal{U}$ and $B_\Phi(t_1)=\Phi(t_1),~B_\Phi(t_\infty)=\Phi(t_\infty),~B_\mathcal{U}(t_1)=\mathcal{U}(t_1),~B_\mathcal{U}(t_\infty)=\mathcal{U}(t_\infty).$
Using Note \ref{Nt3.4}, we have 
\[\Phi^{\alpha}(t_i)=\Phi(t_i),~\mathcal{U}^{\alpha}(t_i)=\mathcal{U}(t_i) \text{ for each } i=1,2,3\ldots \]
From the self-referential equation, 
\[ \Phi^{\alpha}\left(\zeta_j(t)\right)= \Phi\left(\zeta_j(t)\right) + \alpha  \left[\Phi^{\alpha}(t) - B_\Phi(t)\right]\] \\ and, \[\mathcal{U}^{\alpha}\left(\zeta_j(t)\right)= \mathcal{U}\left(\zeta_j(t)\right) + \alpha  \left[\mathcal{U}^{\alpha}(t) - B_\mathcal{U}(t)\right]\]
 for each $ t \in I_j, $ where $j \in J.$ For $t \in \vartheta,$ we deduce \[\Phi^{\alpha}\big(\zeta_j(t)\big) \le \mathcal{U}^{\alpha}\big(\zeta_j(t)\big) \text{ for any } j \in J.\]
 Repeating the above process $k\in \mathbb{N}$ times, we get
 \[\Phi^{\alpha}\left(\zeta_{i_{1}\ldots i_{k}}(t)\right)\le \mathcal{U}^{\alpha}\left(\zeta_{i_{1}\ldots i_{k}}(t)\right) \text{ for any } i_1,\ldots, i_k \in J, ~ t \in \{t_1,t_2,t_3\ldots,t_\infty\},\]
 where $\zeta_{i_{1}\ldots i_{k}}(t)=\zeta_{i_{1}}(\zeta_{i_{2}}(\ldots (\zeta_{i_{k}}(t))))$ and $k \in \mathbb{N}.$\\
 It follows that $\Phi^{\alpha}(t)\le \mathcal{U}^{\alpha}(t)$ for every  $t\in \underset{k \in \mathbb{N}}{\cup}\left(\underset{1\leq i_1,\ldots,i_k< \infty}{\cup}\zeta_{i_{1}\ldots i_{k}}\big(\left\{t_1,t_2,t_3,\ldots,t_\infty\right\}\big)\right).$ \\
Therefore, applying Lemma \ref{lem4.6} and Result \ref{res 4.5}, completes the proof.
\end{proof}

 \begin{definition} \label{newg}
 Consider a set-valued function $\Phi:I\rightarrow \mathrm{K}(\mathbb{R})$. The graph of $\Phi$ is defined as
 \begin{equation}\label{grph}
    \mathcal{G}(\Phi)=\left\{(t,\Phi(t)): \Phi(t)\in \mathrm{K}(\mathbb{R}) \right\}\subset I\times \mathrm{K}(\mathbb{R}).  
   \end{equation}
     Define a metric on the graph of this set valued function; \[d_{\mathcal{G}}((t,\Phi(t)),(w,\Phi(w)))=\lvert t-w\rvert+H_d(\Phi(t),\Phi(w)).\]
   \end{definition}
   Next, we will show that the graph of set valued fractal function $\Phi^{\alpha}$ as defined in (\ref{grph}) is an attractor of an countable iterative function system (CIFS) defined on $I\times \mathrm{K}_c(\mathbb{R})$.\\
Before proving the next theorem. We take note of the following lemma.

 \begin{lemma} \cite{Ver21} \label{lemma5.1}
Let $\Phi\in \mathcal{C}(I, \mathrm{K}_c(\mathbb{R}))$ be a set-valued continuous map and $\Phi^{\alpha}$ be its corresponding $\alpha$-fractal function. Define a function $D:\big(I\times \mathrm{K}_c(\mathbb{R})\big) \times \big(I\times \mathrm{K}_c(\mathbb{R})\big) \rightarrow [0, \infty)$ as
 \[D\big((t,S_1),(w,S_2)\big)=\lvert t-w \rvert+H_d\big(S_1+\Phi^{\alpha}(w),S_2+\Phi^{\alpha}(t)\big).\]
  Then $(I\times \mathrm{K}_c(\mathbb{R}),D)$ is a complete metric space.
 \end{lemma}

 \begin{proposition}\label{prop5.6}
Let $B\in \mathcal{C}(I, \mathrm{K}_c(\mathbb{R}))$ be the base function and $\Phi\in \mathcal{C}(I, \mathrm{K}_c(\mathbb{R}))$ be a set-valued continuous map. For each $j\in J=\{1,2,3,\ldots\}$, define $G_{j}: I \times \mathrm{K}_c(\mathbb{R})\rightarrow I \times \mathrm{K}_c(\mathbb{R})$, such that
 \begin{equation*}\label{cifs}
     \begin{aligned}
 G_{j}(t,S)=\left(\zeta_{j}(t), \alpha S+\Phi(\zeta_{j}(t))-\alpha B(t)\right).
  \end{aligned}
 \end{equation*}
 Then with respect to the metric defined in Lemma \ref{lemma5.1}, each $G_j$ is a contraction map, provided $\max\{\lvert \alpha \rvert,a_j\}<1$ for each $j\in J$.
 \end{proposition}
 
 \begin{proof}
 Let $(t,S_1),(w,S_2)\in I \times \mathrm{K}_c(\mathbb{R})$, then for each $j\in J=\{1,2,3,\ldots\}$, we have
\begin{align*}
      D\big(G_j(t,S_1)&,G_j(w,S_2)\big)\\
     =&D\Big(\big(\zeta_j(t),\alpha S_1+\Phi(\zeta_j(t))-\alpha B(t)\big),\big(\zeta_j(w),\alpha S_2+\Phi(\zeta_j(w))-\alpha B(w)\big)\Big)\\
     =&\big\lvert \zeta_j(t)-\zeta_j(t)\big\rvert + H_d\Big(\alpha S_1+\Phi(\zeta_j(t))-\alpha B(t)+\Phi^{\alpha}(\zeta_j(w)),\\
     &\hspace{5cm}\alpha S_2+\Phi(\zeta_j(w))-\alpha B(w)+\Phi^{\alpha}(\zeta_j(t))\Big)\\
     =&\big\lvert \zeta_j(t)-\zeta_j(w)\big\rvert+ H_d\Big(\alpha S_1+\Phi(\zeta_j(t))-\alpha B(t)+\Phi(\zeta_j(w))+\alpha \Phi^{\alpha}(w)-\alpha B(w),\\
     &\hspace{4cm}\alpha S_2+\Phi(\zeta_j(w))-\alpha B(w)+\Phi(\zeta_j(t))+\alpha \Phi^{\alpha}(t)-\alpha B(t) \Big)\\
     =&a_j \big\lvert t-w \big\rvert+H_d\Big(\alpha S_1+\alpha \Phi^{\alpha}(w),\alpha S_2+\alpha \Phi^{\alpha}(t) \Big)\\
     =&a_j \big\lvert t-w \big\rvert+\lvert \alpha \rvert H_d\Big( S_1+ \Phi^{\alpha}(w), S_2+ \Phi^{\alpha}(t) \Big)\\
     \le& \max\{\lvert \alpha \rvert,a_j\} \Big(\big\lvert t-w \big\rvert+H_d\Big( S_1+ F^{\alpha}(w), S_2+ F^{\alpha}(t) \Big) \Big).
 \end{align*}
 This implies that
 \begin{align*}
  D\big(G_j(t,S_1),G_j(w,S_2)\big) 
     \le&\max\{\lvert \alpha \rvert,a_j\} D\big((t,S_1),(w,S_2)\big).
  \end{align*}
 Hence, each $G_j$ is a contraction mapping, when  $\max\{\lvert \alpha \rvert,a_j\}< 1$. 
 \end{proof}
 \begin{theorem} \label{theorem4.8}
     Assume $\mathcal{I}=\{(I\times \mathrm{K}_c(\mathbb{R})); \mathrm{G_i}: i\in \mathbb{N}\}$, be the CIFS, where $\mathrm{G}_i: I \times\mathrm{K}_c(\mathbb{R}) \rightarrow I\times \mathrm{K}_c(\mathbb{R})$ as defined in Equation \ref{cifs} provided $\max\{\lvert \alpha \rvert,a_j\}<1$ for each $j\in J$  are the contraction mappings defined in Proposition \ref{prop5.6}, then $\mathcal{G}(\Phi^\alpha)=\overline{ \underset{i\in \mathbb{N}}{\bigcup}G_i(\mathcal{G}(\Phi^\alpha))},$ where $\mathcal{G}(\Phi^\alpha)$ denotes the graph of the function $\Phi^\alpha .$
 \end{theorem}
 \begin{proof}
     Since, from the \cite[Lemma $6$]{Ver21} $\big(C(I,\mathrm{K}_c(\mathbb{R})), \mathfrak{d}_c\big)$ is a complete metric space. Also, $C(I,\mathrm{K}_c(\mathbb{R}))$ is a closed subset of $C(I,\mathrm{K}(\mathbb{R}))$. By Theorem \ref{Thm3.2}, $\Phi^\alpha \in C(I,\mathrm{K}_c(\mathbb{R}))$ and $I=\underset{j\in \mathbb{N}}{\bigcup}\zeta_j(I).$\\
     Let $(t,\Phi^\alpha(t)) \in \mathcal{G}(\Phi^\alpha), \text{where }  t \in I$
     \begin{equation*}
          \begin{aligned}
         G_j(t,\Phi^\alpha (t)) &=~(\zeta_j(t), ~\alpha \Phi^\alpha (t)+\Phi(\zeta_j(t))-\alpha B(t))\\
         &=~(\zeta_j(t),\Phi^\alpha (\zeta_j(t))) \in \mathcal{G}(\Phi^\alpha ).
             \end{aligned}
          \end{equation*}
It follows that $G_j(\mathcal{G}(\Phi^\alpha)) \subseteq \mathcal{G}(\Phi^\alpha)$ for all $j\in J=\{1,2,3,\ldots\}$. This gives us that $\overline{\underset{i\in \mathbb{N}}{\bigcup}G_i(\mathcal{G}(\Phi^\alpha))} \subseteq \mathcal{G}(\Phi^\alpha)$. For inverse containment, let $(t,\Phi^\alpha(t)) \in \mathcal{G}(\Phi^\alpha ) $ and $t \in I_n,~n\in \mathbb{N}$, this gives us that 
          \begin{equation*}
              \begin{aligned}
                  (t,\Phi^\alpha (t))&=(\zeta_j(t), \Phi^\alpha (\zeta_j(t)) \text{ for some }j \in J, t\in I_j \\
                  &=~(\zeta_j(t),~\alpha \Phi^\alpha (t)+\Phi(\zeta_j(t))-\alpha B(t))\\ &=~G_j(t,\Phi^\alpha(t))=~G_j(\mathcal{G}(\Phi^\alpha(t) ))\\
                  &\subseteq \overline{ \underset{i\in \mathbb{N}}{\bigcup}G_i(\mathcal{G}(\Phi^\alpha))} .
            \end{aligned}
          \end{equation*}
          If $t=t_\infty$, then we have
          \begin{equation*}
              \begin{aligned}
                  (t_\infty,\Phi^\alpha(t_\infty))&=(t_\infty,\underset{n\rightarrow \infty}{\lim}\Phi^\alpha (t_n)) \\&=\Big(\zeta_j(t_\infty),\big(\underset{n\rightarrow \infty}{\lim}\alpha \Phi^\alpha (t_n)+\Phi(\zeta_j(t_n))-\alpha B(t_n)\big)\Big)\\&= \underset{n\rightarrow \infty}{\lim}\big(\zeta_j(t_n),\alpha \Phi^\alpha (t_n)+\Phi(\zeta_j(t_n))-\alpha B(t_n)\big)\\
                  &=\underset{n\rightarrow \infty}{\lim}G_j(t_n,\Phi^\alpha(t_n)) =\underset{n\rightarrow \infty}{\lim} G_j\big(\mathcal{G}(\Phi^\alpha(t_n))\big) \\
                  &\subseteq \overline{ \underset{i\in \mathbb{N}}{\bigcup}G_i(\mathcal{G}(\Phi^\alpha))}.
              \end{aligned}
          \end{equation*}
          Therefore, we get
          $\mathcal{G}(\Phi^\alpha )=\overline{ \underset{i\in \mathbb{N}}{\bigcup}G_i(\mathcal{G}(\Phi^\alpha))}.$
     \end{proof}
  \section{Invariant Measure }
         In the following section, we have shown the existence of the invariant Borel probability measure supported on the graph of the set-valued fractal function corresponding to the countable data set.
 \begin{theorem}
  Let $\mathcal{I}=\{I\times \mathrm{K}_c(\mathbb{R});G_i:i\in \mathbb{N}\}$ be the CIFS as defined in Proposition \ref{prop5.6}. Assume that $\boldsymbol{p}=(p_{i})_{i\in \mathbb{N}}$ is a probability vector such that $0<p_{i}<1$ for all $i \in\mathbb{N}.$  Then, there exists a unique Borel probability measure $\mu_{\boldsymbol{p}}$ on $I\times \mathrm{K}_c(\mathbb{R}),$ such that
$$\mu_{\boldsymbol{p}}=\sum_{i \in\mathbb{N}}p_{i}\mu_{\boldsymbol{p}}\circ G_{i}^{-1}.$$
Moreover, $\text{supp}(\mu_{\boldsymbol{p}})=\mathcal{G}(\Phi^\alpha)$ and $\text{supp}(\mu_{\boldsymbol{p}})$ denotes the support of the Borel probability measure $\mu_{\boldsymbol{p}}.$
\end{theorem}
\begin{proof}
Let $\mathcal{P}(I\times \mathrm{K}_c(\mathbb{R}))$ denotes the space of all Borel probability measures on $I\times \mathrm{K}_c(\mathbb{R})$. Let $\mu, \nu\in \mathcal{P}(I\times \mathrm{K}_c(\mathbb{R}))$, we define the Monge-Kantorovich metric $d_{MK}$ on the space $\mathcal{P}(I\times \mathrm{K}_c(\mathbb{R}))$ as follows
		$$ d_{MK}(\mu,\nu)=\sup\bigg\{\int_{I\times \mathrm{K}_c(\mathbb{R})}fd\mu-\int_{I\times \mathrm{K}_c(\mathbb{R})}fd\nu: f:I\times \mathrm{K}_c(\mathbb{R})\to\mathbb{R},~~ \text{Lip}(f)\leq1 \bigg\},$$
  where $\text{Lip}(f)$ denotes the Lipschitz constant of the function $f$. The space $(\mathcal{P}(I\times \mathrm{K}_c(\mathbb{R})),d_{MK})$ is a complete metric space, see, for instance, \cite{Bill, Parth}. 
 
  \par Let us define a mapping $\Psi: \mathcal{P}(I\times \mathrm{K}_c(\mathbb{R})) \to \mathcal{P}(I\times \mathrm{K}_c(\mathbb{R}))$ by $\Psi(\mu)=\sum_{i\in \mathbb{N}}p_{i}\mu\circ G_i^{-1}$ and clearly, the mapping $\Psi$ is  well-defined. Firstly, we will show that the mapping $\Psi$ we have defined is a contraction. With the notation $c_j=\max \{\lvert \alpha \rvert,a_j\}$ and $c_{max}= \max_{j \in \mathbb{N}} c_j$, we have
  \begin{equation*}
              \begin{aligned}
               d_{MK}(\Psi(\mu),\Psi(\nu)) =& \sup_{\text{Lip}(f)\leq1}\bigg\{\int_{I\times \mathrm{K}_c(\mathbb{R})}f(t)d(\Psi(\mu))(t)-\int_{I\times \mathrm{K}_c(\mathbb{R})}f(t)d(\Psi(\nu))(t)\bigg\}\\ =& \sup_{\text{Lip}(f)\leq1}\bigg\{\int_{I\times \mathrm{K}_c(\mathbb{R})}f(t)d(\sum_{i\in \mathbb{N}}p_{i}\mu\circ G_i^{-1})(t)\\& \hspace{4cm}-\int_{I\times \mathrm{K}_c(\mathbb{R})}f(t)d(\sum_{i\in \mathbb{N}}p_{i}\nu\circ G_i^{-1})(t)\bigg\}\\
               =& \sup_{\text{Lip}(f)\leq1}\bigg\{\sum_{i\in \mathbb{N}}p_{i} \bigg(\int_{I\times \mathrm{K}_c(\mathbb{R})}f(t)d(\mu\circ G_i^{-1})(t)\\& \hspace{4cm} -\int_{I\times \mathrm{K}_c(\mathbb{R})}f(t)d(\nu\circ G_i^{-1})(t)\bigg)\bigg\}\\ 
               =& \sup_{\text{Lip}(f)\leq1}\bigg\{\sum_{i\in \mathbb{N}}p_{i}\bigg( \int_{I\times \mathrm{K}_c(\mathbb{R})}f(G_i(w))d(\mu)(w) \\& \hspace{4cm} -\int_{I\times \mathrm{K}_c(\mathbb{R})}f(G_i(w))d(\nu)(w)\bigg)\bigg\}\\
               \end{aligned}
          \end{equation*}
               \begin{equation*}
              \begin{aligned}
             \hspace{2.5cm}
          =& \sup_{\text{Lip}(f)\leq1}\bigg\{\sum_{i\in \mathbb{N}}p_{i} c_i \bigg(\int_{I\times \mathrm{K}_c(\mathbb{R})}\frac{1}{c_i}f(G_i(w))d(\mu)(w) \\& \hspace{4cm} -\int_{I\times \mathrm{K}_c(\mathbb{R})}\frac{1}{c_i}f(G_i(w))d(\nu)(w)\bigg)\bigg\}\\ 
            \le &  c_{max} \sum_{i\in \mathbb{N}}p_{i} \sup_{\text{Lip}(g)\leq1}\bigg\{ \int_{I\times \mathrm{K}_c(\mathbb{R})}gd(\mu)-\int_{I\times \mathrm{K}_c(\mathbb{R})}gd(\nu)\bigg\}\\\le& c_{max} d_{MK}(\mu,\nu),
              \end{aligned}
          \end{equation*}
where $g(w)=\frac{1}{c_i}f(G_i(w))$ is a Lipschitz function with $\text{Lip}(g)\leq1$. By the Banach fixed point theorem, we can conclude that there exist a unique Borel measure $\mu_{\boldsymbol{p}}$ and satisfying the following equation $\mu_{\boldsymbol{p}}=\sum_{i \in\mathbb{N}}p_{i}\mu_{\boldsymbol{p}}\circ G_{i}^{-1}.$
Next, we will show that that $\text{supp}(\mu_{\boldsymbol{p}})= \mathcal{G}(\Phi^\alpha).$ We will prove this by showing that $\text{supp}(\mu_{\boldsymbol{p}})$ is the attractor of the CIFS. Now the
\begin{align*}
\mu_{\boldsymbol{p}}\bigg(\overline{\bigcup_{i
  \in \mathbb{N}}G_i(\text{supp}(\mu_{\boldsymbol{p}})})\bigg)&=\sum_{i\in\mathbb{N}}p_{i}\mu_{\boldsymbol{p}}\circ G_i^{-1} \bigg(\overline{\bigcup_{i
  \in \mathbb{N}}G_i(\text{supp}(\mu_{\boldsymbol{p}})})\bigg)\\&\geq \sum_{i\in\mathbb{N}}p_{i}\mu_{\boldsymbol{p}}\circ G_i^{-1}(G_i(\text{supp}(\mu_{\boldsymbol{p}}))\\&=\sum_{i\in\mathbb{N}}p_{i}=1.
\end{align*}
This implies that $\text{supp} (\mu_{\boldsymbol{p}})\subseteq \overline{\bigcup_{i
  \in \mathbb{N}}G_i(\text{supp}(\mu_{\boldsymbol{p}})}).$ And, also
  \begin{align*}
 1=\mu_{\boldsymbol{p}}(\text{supp}(\mu_{\boldsymbol{p}}))= \sum_{i\in\mathbb{N}}p_{i}\mu_{\boldsymbol{p}}\circ G_i^{-1} (\text{supp}(\mu_{\boldsymbol{p}}))\leq \sum_{i\in\mathbb{N}}p_{i}\mu_{\boldsymbol{p}}(I\times \mathrm{K}_c(\mathbb{R}))=1.  
  \end{align*}
Thus, we get $\sum_{i\in\mathbb{N}}p_{i}\mu_{\boldsymbol{p}}\circ G_i^{-1} (\text{supp}(\mu_{\boldsymbol{p}}))=1$. Since $\sum_{i \in \mathbb{N}}p_{i}=1,$ we obtain 
\begin{align*}
    \text{supp}(\mu_{\boldsymbol{p}}) \subseteq  G_i^{-1} (\text{supp}(\mu_{\boldsymbol{p}}))\text{ for each } i \in \mathbb{N}.
\end{align*}
This implies that  $\overline{\bigcup_{i
  \in \mathbb{N}}G_i(\text{supp}(\mu_{\boldsymbol{p}})})\subseteq \text{supp}(\mu_{\boldsymbol{p}}),$ which proves our claim. By Theorem \ref{theorem4.8}, we have $\mathcal{G}(\Phi^\alpha)=\text{supp}(\mu_{\boldsymbol{p}}),$ which is the required result. 
  \end{proof}
We note the coming remark to highlight the issue of not considering the case when all $p_i=0$ and $c_i=0$ for studying invariant measures in fractal geometry literature.
 \begin{remark}
     In the first case, consider $p_1=1$ and remaining $p_i=0,$ for all $i\in \mathbb{N} \setminus\{1\}.$ Clearly, $\mu_{\boldsymbol{p}}=\delta_{t_0},$ a Dirac measure supported on $t_0.$ Here, $t_0$ is the fixed point of $G_1.$ Secondly, if some $p_i$ are zero, for example, $p_1=0$ and others are positive then $\mu_{\boldsymbol{p}}$ will be supported on $A^*$, where $A^* =\overline{\cup_{i=2}^{\infty} G_i(A^*)}$.  Third, consider for each $i,\ f_i:\mathbb{R}\rightarrow\mathbb{R}$ the constant map given as $f_i(t)=t_i, \forall \ t\in \mathbb{R},$ and $p_i>0.$ Then, it is straightforward to verify that $\mu_{\boldsymbol{p}}$ is the convex combination of Dirac measures supported on $\{t_1,t_2,t_3,\dots,t_{\infty}\}$, that is, $\mu_{\boldsymbol{p}}=\sum_{i=1}^{\infty}p_i\delta_{t_i}.$ These previous cases convince us to consider $p_i>0$ and $c_i>0$ for all $i \in \mathbb{N}$ to avoid the above trivial cases. Further, for $p_i=c_i^s,$ then in case of similarity IFS over the finite data set under the OSC, \cite{Moran} confirms that $\mu_{\boldsymbol{p}}$ is equal to normalized $s$-dimensional Hausdorff measure supported on the attractor of the associated IFS. 
     
 \end{remark}

 \section{Some Dimensional Results}
 In the present section, we have determined the upper and lower bounds on the Hausdorff dimension of the graph of the set-valued fractal function corresponding to the countable data set.

   \begin{proposition}\label{prop6}
    If~ $ \sum_{i=1}^{\infty}b_i^{s_0} =1$ for some $s_0 \ge 0$, then $s_k$ converges to $s_0$, where $\sum_{i=1}^{k}b_i^{s_k} =1$, that is, $s_0= \sup_{k\geq 2}s_k =\lim_{k \to \infty} s_k.$
\end{proposition}
\begin{proof}
    Since $ 0<b_i<1$ for all $i\in \mathbb{N}$. Given $\sum_{i=1}^{\infty}b_i^{s_0} =1$ and $\sum_{i=1}^{k}b_i^{s_k} =1$ where $s_k$ is uniquely determined for $k\geq 2$. Since, we also have $\sum_{i=1}^{k+1}b_i^{s_{k+1}} =1$ which gives us $\sum_{i=1}^{k}b_i^{s_{k+1}} <1$ as the $b_{k+1}^{s_{k+1}}>0$.
    From,  $\sum_{i=1}^{\infty}b_i^{s_0} =1$ and $\sum_{i=1}^{k}b_i^{s_{k+1}} <1$, we can conclude that $s_k\leq s_{k+1}$, that is, $s_k$ for $k\geq 2$ is a monotonically increasing sequence. \\ Further, as $\sum_{i=1}^{\infty}b_i^{s_0} =1$ then for any $s > s_0$, we have $\sum_{i=1}^{\infty}b_i^{s}<1$. Thus, $s_0$ is the lowest upper bound of $s_k$.\\ Therefore, we obtain \[ s_0= \sup_{k\geq 2}s_k =\lim_{k \to \infty} s_k. \]
\end{proof}
\begin{theorem}
 Consider $ \{I\times \mathrm{K}_c(\mathbb{R});G_i:i\in \mathbb{N}\}$ be the CIFS, as defined in Theorem \ref{theorem4.8}, satisfying the 
 \[b_id_\mathcal{G}((t,S_1),(w,S_2))\leq d_\mathcal{G}(G_i(t,S_1),G_i(w,S_2)))\le  c_id_\mathcal{G}((t,S_1),(w,S_2))\] for every $(t,S_1),(w,S_2)\in I \times \mathrm{K_c(\mathbb{R})},$ such that $0<b_i\le c_i < 1$ for all $i \in \mathbb{N}.$ Then $s_*\le \text{dim}_H(\mathcal{G}(\Phi^\alpha))\le\text{dim}_B(\mathcal{G}(\Phi^\alpha))\le s^*,$ where $s_*=\sup_{k\geq 2}s_k$ and each of $s_k$ is uniquely defined by the $\sum_{i=1}^{k}b_i^{s_k} =1$ and $s^*=\text{max}\bigg\{ \text{inf } \big\{s: ~ \sum_{i\in \mathbb{N}}c_{i}^s \le 1 \big\},1\bigg\}.$
\end{theorem}
\begin{proof}
    In Theorem \ref{theorem4.8} we proved that $\mathcal{G}(\Phi^\alpha )$ is the attractor of the CIFS , that is 
    \[ \mathcal{G}(\Phi^\alpha ) = \overline{ \underset{i\in \mathbb{N}}{\bigcup}G_i(\mathcal{G}(\Phi^\alpha))}. \]
Since, 
\[ d_{\mathcal{G}}( G_i(t,S_1),G_j(w,S_2)\le c_i d_\mathcal{G}((t,S_1),(w,S_2)) \text{ and } \sup_{n\in \mathbb{N}}c_n <1 \]

Let $O=(t_j,t_{j+1})\times \mathrm{K}_c(\mathbb{R})$. Then, $G_{i_1}(O)\cap G_{i_2}(O) = \phi$ for $i_1 \neq i_2$ and also $G_{i_1}(O)\subseteq O$ and $O \cap \mathcal{G}(\Phi^\alpha)\neq \phi .$ Hence, CIFS satisfies the SOSC. \\
For each $k\geq 2$ we considered the finite IFS $\{I\times \mathrm{K}_c(\mathbb{R});G_i:i\in \{1,2,\dots k\}\}$ with help of CIFS. Let $A_k$ be the attractor of the finite IFS. Firstly, we show the SOSC for the finite IFS. Clearly, \[ G_{j_1}(O)\cap G_{j_2}(O)= \phi~~ \forall ~~
 j_1\neq j_2 \text{ and } G_j(O) \subseteq O ~~ \forall ~~ j \in \{1,2, \dots,k\}.\]
 It remains to show that $O\cap A_k \neq \phi $. To prove this, consider the contradiction such that $O\cap A_k = \phi$. Since $O=(t_j,t_{j+1})\times \mathrm{K}_c(\mathbb{R})$ and $O\cap A_k = \phi $, $A_k$ contains element of the type $(t_n,S)$ for $n \in \{j,j+1\} $ and $S\in \mathrm{K}_c(\mathbb{R}) $. Let $(t_j,S) \in A_k$ for some $S \in \mathrm{K}_c(\mathbb{R})$. Since $n \geq 2,$ we have $G_2$ in IFS, where $G_2(t,S)=(\zeta_2(t),\alpha S+\Phi(\zeta_j(t))-\alpha B(t))$. The map $G_2:[t_j,t_{j+1}]\times \mathrm{K}_c(\mathbb{R}) \rightarrow [t_j,t_{j+1}]\times \mathrm{K}_c(\mathbb{R})$ and $A_k=\cup_{i=1}^{k}G_i(A_k)$ this gives us that $G_2(A_2)\subseteq A_2$. 
 \begin{equation*}
     \begin{aligned}
         G_2(t_j,S)=&~ (\zeta_2(t_j),\alpha S+\Phi(\zeta_2(t_j))-\alpha B(t_j)) \\ =&~(t',S') \text{ for some } t'\in O \text{ and } S' \in \mathrm{K}_c(\mathbb{R}).
     \end{aligned}
 \end{equation*}
 This implies that $(t',S')$ is in $A_k$, which is a contradiction. Similarly for $(t_{j+1},S)\in A_k,$ we get the contradiction. Thus $O\cap A_k \neq \phi $. Therefore, for each, $n\geq 2$, finite IFS satisfies the SOSC. In view of \cite{Mverma4} we have  $\dim_H(A_k)\ge s_k$, where $s_k$ is uniquely defined by the $\sum_{i=1}^{k}b_i^{s_k} =1$. It is easy to see that $A_k \subseteq \mathcal{G}(\Phi^\alpha)$. Using the monotonic property of Hausdorff dimension we obtain $ \dim_H(\mathcal{G}(\Phi^\alpha)) \geq s_k,$ for all $k \geq 2$.\\
 This gives us the lower bound that $s_* \leq \dim_H(\mathcal{G}(\Phi^\alpha)),$ where $s_*=\sup_{k\geq 2}s_k$ and each of $s_k$ is defined by the $\sum_{i=1}^{k}b_i^{s_k} =1$.\\
For the upper bound on the dimension of graph of fractal function, let $\mathcal{G}_o(\Phi^\alpha )=\underset{i\in \mathbb{N}}{\bigcup}G_i(\mathcal{G}(\Phi^\alpha)).$ Now,
\begin{equation*}
    \begin{aligned}
     \lvert\mathcal{G}_o(\Phi^\alpha)\rvert= \bigg\lvert\underset{i_1,i_2\in \mathbb{N}}{\bigcup}G_{i_1}\circ G_{i_2}(\mathcal{G}(\Phi^\alpha))\bigg\rvert=&\bigg\lvert\bigcup\limits_{\substack{i_j\in \mathbb{N} \\ 1\leq j\le N}}G_{i_1}\circ G_{i_2}\dots \circ G_{i_N}(\mathcal{G}(\Phi^\alpha))\bigg\rvert \\ \leq & \sum_{\substack{i_j\in \mathbb{N} \\ 1\leq j\le N}}c_{i_1}c_{i_2} \dots c_{i_N}\lvert\mathcal{G}(\Phi^\alpha)\rvert.
    \end{aligned}
\end{equation*}
 For any $\delta > 0$, we can choose the $N$ to be sufficiently large such that $\lvert\mathcal{G}_o(\Phi^\alpha)\rvert < \delta$. Further
 \begin{equation*}
     \begin{aligned}
         H_\delta^s(\mathcal{G}_o(\Phi^\alpha)) \le &\sum_{\substack{i_j\in \mathbb{N} \\ 1\leq j\le N}}c_{i_1}^sc_{i_2}^s \dots c_{i_N}^s\lvert\mathcal{G}(\Phi^\alpha)\rvert^s  \\ =& \lvert\mathcal{G}(\Phi^\alpha)\rvert^s \sum_{i_1\in \mathbb{N}}c_{i_1}^s\sum_{i_2\in \mathbb{N}}c_{i_2}^s\dots\sum_{i_N\in \mathbb{N}}c_{i_N}^s \\  \le & \lvert\mathcal{G}(\Phi^\alpha)\rvert^s\bigg(\sum_{i\in \mathbb{N}}c_{i}^s\bigg)^N .
     \end{aligned}
 \end{equation*}
 Taking $\delta \rightarrow 0$, we have convergence if and only if $\sum_{i\in \mathbb{N}}c_{i}^s < 1$. Hence $\text{dim}_H(\mathcal{G}_o(\Phi^\alpha))\le \text{inf}\bigg\{s: ~ \sum_{i\in \mathbb{N}}c_{i}^s \le 1 \bigg\}.$ As the $\mathcal{G}(\Phi^\alpha)=\mathcal{G}_o(\Phi^\alpha)\cup \{(t_\infty,\Phi^\alpha(t_\infty))\}$, where the set $ \{(t_\infty,\Phi^\alpha(t_\infty))\}$ has dimension zero. By using the finite stability property of Hausdorff dimension, we have $\text{dim}_H(\mathcal{G}(\Phi^\alpha))\le s^*$, where   $s^*=\text{ max }\bigg\{\text{inf}\big\{s: ~ \sum_{i\in \mathbb{N}}c_{i}^s \le 1 \big\},1\bigg\}.$
 \end{proof}
 \begin{remark}
     The upper bound of Hausdorff dimension is more complicated in CIFS. We need to investigate the set $\overline{ \underset{i\in \mathbb{N}}{\bigcup}G_i(K)} ~~\setminus~
     \underset{i\in \mathbb{N}}{\bigcup}G_i(K)$, where $K$ is the attractor of the CIFS. If both sets coincide except for the countable many points, then by the countable stability property we will have the $\text{dim}_H(\underset{i\in \mathbb{N}}{\bigcup} G_i(K) )=\text{dim}_H\big( \overline{ \underset{i\in \mathbb{N}}{\bigcup}G_i(K)}\big).$ 
     \end{remark}

\begin{example}  Consider a set-valued map (SVM) $\Phi:[a,b]\rightrightarrows \mathbb{R}$. As the graph of SVM is given according to \eqref{Gf1}. If 
\begin{enumerate}
     \item $\Phi_1(t)=[t^2+1,t^2+2]$, then the graph $G_{\Phi_{1}}$ of this SVM is an interval band in $\mathbb{R}^2$ and $\dim_H(G_{\Phi_{1}})=2$.
    \item $\Phi_2(t)=\{1\}$, then the graph $G_{\Phi_{2}}$ of this SVM is a straight line in $\mathbb{R}^2$ and $\dim_H(G_{\Phi_{2}})=1$.

\end{enumerate}
    
\end{example}
\begin{remark}
    
If we have a SVM $\Phi : [a,b] \to \mathrm{K}(\mathbb{R})$ and a
function $F : [a,b] \to \mathbb{R}$, where $F(x) \in \Phi(x)$ for all $x\in [a,b]$, then $F$ is called a selection of $\Phi$. It is noteworthy that for any $1 \le s \le 2$, there exists a selection $F_{s} : [a,b] \to \mathbb{R}$ of the set-valued map $\Phi_1$ mentioned in the above example such that $\dim_H(G_{F_{s}})=s.$
\end{remark}

\section*{Acknowledgements}

 The first author is financially supported by the University Grant Commission, India, in the form of a junior research fellowship.

\bibliographystyle{amsplain}

\end{document}